\documentclass[12pt]{article}
\usepackage[margin=1in]{geometry}

\usepackage{geometry}
\usepackage{amsmath}
\usepackage{amssymb}
\usepackage{amsthm}
\usepackage{enumitem}
\usepackage{tikz}
\usepackage{graphicx}
\usepackage{caption}
\usepackage{subcaption}
\usepackage{url}

\newcommand{\X}{{\mathcal X}}
\newcommand{\PP}{{\mathcal P}}
\newcommand{\F}{{\mathcal F}}

\newcommand{\poly}{{\text{poly}}}

\newtheorem{theorem}{Theorem}[section]
\newtheorem{lemma}[theorem]{Lemma}

\newtheorem{claim}{Claim}[theorem]
\newtheorem{conjecture}[theorem]{Conjecture}
\newtheorem{question}[theorem]{Question}

\newenvironment{subproof}[1][Proof]{\begin{proof}[#1]}{\end{proof}}

\title{Tree-width of a graph excluding an apex-forest or a wheel as a minor}

\author{Chun-Hung Liu\thanks{chliu@tamu.edu. Department of Mathematics, Texas A\&M University, USA. Partially supported by NSF under CAREER award DMS-2144042.}
\and 
Youngho Yoo\thanks{yyoo2@alaska.edu. Department of Mathematics and Statistics, University of Alaska Fairbanks, USA.}}
\date{\today}

\begin{document}

\maketitle

\abstract{
The Grid Minor Theorem states that for every planar graph $H$, there exists a smallest integer $f(H)$ such that every graph with tree-width at least $f(H)$ contains $H$ as a minor. The only known lower bounds on $f(H)$ beyond the trivial bound $f(H)\geq |V(H)|-1$ come from the maximum number of disjoint cycles in $H$.
In this paper, we study $f(H)$ for planar graphs $H$ with no two disjoint cycles.
We prove that $f(H)=|V(H)|-1$ for every apex-forest $H$. 
This result improves a bound of Leaf and Seymour and contains all known large graphs $H$ meeting the trivial lower bound to our knowledge.
We also prove that $f(H)\leq \max\{\tfrac32|V(H)|-\tfrac92,|V(H)|-1\}$ for every wheel $H$.
}

\section{Introduction}

Tree-width is a measure of a graph's proximity to a tree and is fundamental to graph structure and algorithms. 
The Grid Minor Theorem of Robertson and Seymour \cite{robertson1991graph} states that for every positive integer $k$, every graph with sufficiently large tree-width contains the $k\times k$-grid $G_k$ as a minor.
The same conclusion holds for every planar graph $H$ since $H$ is a minor of $G_k$ for some $k=O(|V(H)|)$ as shown by Robertson, Seymour, and Thomas \cite{robertson1994quickly}. On the other hand, such a conclusion does not hold for any nonplanar graph $H$ because $G_k$ has tree-width $k$ but does not contain $H$ as a minor for any $k$ since $G_k$ is planar. 

For a planar graph $H$, let $f(H)$ denote the smallest integer such that every graph with tree-width at least $f(H)$ contains $H$ as a minor. 
While the original bound on $f(G_k)$ from the Grid Minor Theorem  \cite{robertson1991graph} was enormous, significant strides have been made over the years towards tighter bounds \cite{chekuri2016polynomial, chuzhoy2015excluded, chuzhoy2016improved, chuzhoy2021towards, kawarabayashi2020linear, leaf2015tree, robertson1994quickly} and currently the best known upper bound is $f(G_k)=O(k^9\poly\log k)$ by Chuzhoy and Tan \cite{chuzhoy2021towards}. Hence, for every planar graph $H$, we have $f(H)=O(|V(H)|^9\poly\log|V(H)|)$.
The best known lower bound on $f(G_k)$ comes from the existence of expander graphs $G$ with tree-width $\Omega(|V(G)|)$ and girth $\Omega(\log|V(G)|)$ as observed in \cite{robertson1994quickly}, which implies that if a planar graph $H$ has $t$ (vertex-)disjoint cycles, then $f(H)=\Omega(t\log t)$. 
In particular, this gives $f(G_k)=\Omega(k^2\log k^2)=\Omega(|V(G_k)|\log |V(G_k)|)$.

There is also a trivial lower bound $f(H)\geq |V(H)|-1$ for every planar graph $H$; the complete graph on $|V(H)|-1$ vertices has tree-width $|V(H)|-2$ but does not contain $H$ as a minor.
This trivial lower bound has a connection to Hadwiger's conjecture on graph coloring.
Seymour \cite{seymour2016hadwiger} asked to ascertain the graphs $H$ for which every graph with no $H$ minor is properly $(|V(H)|-1)$-colorable.
Hadwiger's conjecture is equivalent to the statement that all graphs are positive answers to Seymour's question.
Since graphs of tree-width at most $w$ are $w$-degenerate and hence $(w+1)$-colorable, if $f(H) = |V(H)|-1$, then every graph with no $H$ minor is properly $(|V(H)|-1)$-colorable and provides a positive answer to Seymour's question.

This paper is a step towards a characterization of the graphs $H$ with $f(H) = |V(H)|-1$.
Some small graphs $H$ are known to satisfy $f(H)=|V(H)|-1$.
Trivially, we have $f(K_t)=|V(K_t)|-1$ for $t \leq 3$.
It is well-known (following from a result of Dirac \cite{dirac1952property}) that $f(K_4)=3=|V(K_4)|-1$, and Dieng \cite{dieng} showed that $f(K_{2,4})=5=|V(K_{2,4})|-1$. 
These results imply that every graph $H$ that has at most four vertices or is a spanning subgraph of $K_{2,4}$ satisfies $f(H)=|V(H)|-1$.

In addition, existing literature on the structure of graphs excluding a small graph as a minor provide several other instances of graphs $H$ with $f(H)=|V(H)|-1$.
We assume that all graphs are simple unless explicitly stated otherwise.
For a graph $H$, suppose $\F(H)$ is a set of graphs such that every $H$ minor free graph has a tree-decomposition such that every torso is a minor of a graph in $\F(H)$. Then the tree-width of every $H$ minor free graph is at most the maximum tree-width of a graph in $\F(H)$; in particular, if every graph in $\F(H)$ has tree-width at most $|V(H)|-2$, then $f(H)=|V(H)|-1$. Such sets $\F(H)$ are known for the following graphs $H$:
    \begin{itemize}
        \item $H=K_5^-$, where $K_t^-$ is the graph obtained from $K_t$ by deleting an edge \cite[Theorem 3.3]{DING2013355} (originally due to \cite{Wagner});
        
        \item $H$ is the graph obtained from the prism by adding an edge, where the \emph{prism} is the graph obtained from two disjoint triangles by adding a perfect matching between them \cite[Theorem 3.6]{DING2013355};

        \item $H$ is the graph obtained from the 5-wheel by adding an edge, where the \emph{$k$-wheel} is the graph obtained from a cycle of length $k$ by adding a vertex adjacent to all other vertices  \cite[Theorem 4.4]{DING2013355};

        \item $H$ is the graph obtained from the cube by contracting an edge, where the \emph{cube} is the planar graph obtained from two disjoint 4-cycles by adding a perfect matching between them \cite[Theorem 4.4]{DingLiu2012};
                
        \item $H$ is the graph obtained from the disjoint union of $K_3$ and $K_4^-$ by adding a matching of size three between them so that a vertex of degree 3 in $K_4^-$ is unmatched \cite[Theorem 4.6]{DING2013355};
        
        \item $H$ is the cube \cite[Theorems 11.1-11.4]{maharry2000characterization}; and
        
        \item $H$ is the octahedron (i.e.~$H=K_{2,2,2}$)  \cite[Theorem 1.1]{DingOctahedron}.
        
    \end{itemize} 
    
    For each graph $H$ listed above, the cited theorem describes the corresponding set $\F(H)$ of graphs.
    The cited papers do not explicitly discuss the relation between $\F(H)$ and tree-width, but it can be readily checked (we omit these proofs) that the graphs in $\F(H)$ have tree-width at most $|V(H)|-2$ by constructing tree-decompositions of width at most $|V(H)|-2$. 
    Hence, these graphs $H$ (and their spanning subgraphs) satisfy $f(H)=|V(H)|-1$.
Note that a proof of $f(K_5^-)=4$ can also be obtained by using the characterization of graphs of tree-width at most three \cite{ARNBORG19901,Satyanarayana}.
Beyond these small graphs $H$, it is known that $f(H)=|V(H)|-1$ if $H$ is a forest \cite{bienstock1991quickly, diestel1995graph, seymour2023shorter} or a cycle \cite{birmele2003tree, fellows1989search}.

Recall that if $H$ is a planar graph with $c$ disjoint cycles, then $f(H)=\Omega(c\log c)$, so the trivial lower bound $f(H)\geq |V(H)|-1$ cannot be attained by a graph $H$ with many (i.e.~more than $\Omega(\frac{|V(H)|}{\log|V(H)|})$) disjoint cycles.
It is thus natural to study $f(H)$ for planar graphs $H$ that do not contain many disjoint cycles.

Since forests $H$ (i.e.~graphs with no cycles) are already known to satisfy $f(H)=|V(H)|-1$, we address the next step of studying $f(H)$ for planar graphs $H$ with no two disjoint cycles.
Lov\'asz \cite{lovasz1965graphs} (see \cite{bollobas2004}) characterized the graphs with no two disjoint cycles; the 2-connected planar graphs among them are \emph{apex-forests} (graphs that can be made a forest by deleting at most one vertex)
 and subgraphs of subdivisions of the following multigraphs: the multigraph obtained from the wheel by duplicating any number of times the edges incident with one vertex adjacent to every other vertex, and the multigraph obtained from $K_{5}^-$ by duplicating any number of times the edges of the triangle induced by the three vertices of degree 4.  

As the first contribution of this paper, we establish that $f(H)$ meets the trivial lower bound for all apex-forests $H$.

\begin{theorem} \label{thm:apexforest}
Let $H$ be an apex-forest. Then every graph with tree-width at least $|V(H)|-1$ contains $H$ as a minor. In other words, $f(H) = |V(H)|-1$.
\end{theorem}

Theorem \ref{thm:apexforest} generalizes the known results that $f(H)=|V(H)|-1$ when $H$ is a tree or a cycle, which are the only classes of graphs $H$ with $|V(H)|\geq 9$ previously known to satisfy $f(H)=|V(H)|-1$, to our knowledge.
Theorem \ref{thm:apexforest} also improves the previously best-known bound $f(H)\leq \frac32|V(H)|-2$ for apex-forests $H$ by Leaf and Seymour \cite{leaf2015tree}.
As stated above, every graph $H$ with $f(H)=|V(H)|-1$ gives an optimal bound for the degeneracy of $H$ minor free graphs and gives a positive answer to Seymour's question on graph coloring.
So Theorem \ref{thm:apexforest} implies that if $H$ is an apex-forest, then every $H$ minor free graph is $(|V(H)|-2)$-degenerate, which recovers a special case of a result of Liu and Yoo \cite{liu2024tight} on the degeneracy of graphs that do not contain a ``contractibly orderable graph'' as a minor.
Moreover, by \cite{BODLAENDER199792}, for every positive integer $k$, every graph that allows $k$-label Interval Routing Schemes under dynamic cost edges is $K_{2,2k+1}$ minor free and hence has tree-width at most $f(K_{2,2k+1})-1 \leq 2k+1$ by Theorem \ref{thm:apexforest}.
We omit the definitions and background related to $k$-label Interval Routing Schemes and refer interested readers to \cite{BODLAENDER199792}.

Recall that the 2-connected planar graphs $H$ with no two disjoint cycles other than apex-forests are subgraphs of subdivisions of multigraphs obtained from $K_5^-$ or a wheel by duplicating certain edges, and that $f(K_5^-) = |V(K_5^-)|-1$.

The second contribution of this paper provides a new upper bound for $f(H)$ when $H$ is a wheel. 

\begin{theorem} \label{thm:wheel}
Let $H$ be a wheel.
Then every graph with tree-width at least $\max\{\frac32|V(H)|-\frac{9}{2},|V(H)|-1\}$ contains $H$ as a minor.
In other words, $f(H) \leq \max\{\frac32|V(H)|-\frac{9}{2},|V(H)|-1\}$, and if $|V(H)| \leq 7$, then $f(H)=|V(H)|-1$.
\end{theorem}

Theorem \ref{thm:wheel} improves the earlier bound $f(H)\leq 36|V(H)|-38$ of Raymond and Thilikos \cite{raymond2017low} and the more recent bound $f(H)\leq 2|V(H)|+18\left\lceil \frac{1+\sqrt{2|V(H)|-1}}{4}\right\rceil - 9$ by Gollin, Hendrey, Oum, and Reed \cite{gollin2024linear}.
As $K_4$ is a wheel, Theorem \ref{thm:wheel} recovers the classical result that every $K_4$ minor free graph has tree-width at most two.

In light of Lov\'asz's characterization and the fact that every connected planar graph with no two disjoint cycles can be obtained from a 2-connected one by repeatedly adding leaves, we pose the following questions.
\begin{question}\label{q:generateH'}
Let $H$ be a planar graph. 
\begin{enumerate}
    \item If $H'$ is obtained from $H$ by subdividing an edge of $H$ once, then is it true that $f(H')=f(H)+1$?
    \item If $H'$ is obtained from $H$ by duplicating an edge of $H$ once and subdividing the new edge once, and if the maximum number of disjoint cycles in $H'$ is the same as that in $H$, then is it true that $f(H')=f(H)+1$?
    \item If $H'$ is obtained from $H$ by adding a new vertex adjacent to one vertex in $H$, then is it true that $f(H')=f(H)+1$?
\end{enumerate}
\end{question}
If the answers to these three questions are all positive, and if Theorem \ref{thm:wheel} could be improved to $f(H)\leq |V(H)|-1$ for all wheels $H$, then Lov\'asz's characterization (together with Theorem \ref{thm:apexforest} and the result $f(K_5^-)=4$) would imply that $f(H)=|V(H)|-1$ for every planar graph $H$ with no two disjoint cycles. 

Since the only known non-trivial lower bounds on $f(H)$ come from graphs $H$ with many disjoint cycles, a natural question is whether the excess factor of $f(H)$ beyond the trivial lower bound $f(H)\geq |V(H)|-1$ is bounded by a function of the maximum number of disjoint cycles in $H$. 

\begin{conjecture} \label{conj:boundg(c)}
    For every nonnegative integer $c$, there exists an integer $g(c)$ such that $f(H)\leq g(c)\cdot |V(H)|$ for every planar graph $H$ with no $c+1$ disjoint cycles. 
\end{conjecture}

Note that Conjecture \ref{conj:boundg(c)} was also independently proposed by Bruce Reed and David Wood.

Recall that $f(H)=|V(H)|-1$ for every forest $H$, so Conjecture \ref{conj:boundg(c)} is true for $c=0$ with $g(c)=1$.
A possible approach for the case $c=1$ would be to prove weaker forms of the three questions in Question \ref{q:generateH'} allowing $f(H')=f(H)+O(1)$. 

Conjecture \ref{conj:boundg(c)} has recently been shown to be true when $H$ is the disjoint union of $c$ cycles. Note that the case $c=1$ follows from the known result that $f(H)=|V(H)|-1$ for every cycle $H$. 
Gollin, Hendrey, Oum, and Reed \cite{gollin2024linear} showed that if $c=2$, then $f(H)=(1+o(1))|V(H)|$ and that for all $c\geq 3$, $f(H)=\frac32|V(H)|  +O(c^2\log c)$.
Hatzel, Liu, Reed, and Wiederrecht \cite{hatzel2025bounds} showed that for every $c \geq 2$, $f(H)= O(\log c\cdot |V(H)| + c\log c\cdot \log |V(H)|)$ which implies that Conjecture \ref{conj:boundg(c)} holds for disjoint unions of cycles with $g(c)=O(c\log c)$. 

\section{Preliminaries}

Let $G$ be a graph.
A {\it tree-decomposition} of $G$ is a pair $(T,\X)$ where $T$ is a tree and $\X=(X_t: t \in V(T))$ is a collection of sets such that
    \begin{itemize}
        \item $\bigcup_{t \in V(T)}X_t=V(G)$,
        \item for every $e \in E(G)$, there exists $t \in V(T)$ such that $X_t$ contains both ends of $e$, and
        \item for every $v \in V(G)$, the set $\{t \in V(T): v \in X_t\}$ induces a connected subgraph of $T$.
    \end{itemize}
For every $t \in V(T)$, the set $X_t$ is called the {\it bag} at $t$.
The {\it width} of $(T,\X)$ is $\min_{t \in V(T)}|X_t|-1$.
The {\it tree-width} of $G$ is the minimum width of a tree-decomposition of $G$.
A \emph{path-decomposition} is a tree-decomposition $(T,\X)$ such that $T$ is a path.

A {\it rooted tree} $T$ is a tree with a specified vertex called the {\it root}.
For every $t \in V(T)$, a {\it descendant} of $t$ is a vertex $t'$ of $T$ such that the unique path in $T$ between $t'$ and the root contains $t$.
Note that every vertex of $T$ is a descendant of itself.
A {\it proper descendant} of $t$ is a descendant of $t$ not equal to $t$.
Every proper descendant of $t$ adjacent to $t$ is called a {\it child} of $t$.
For every $t \in V(T)$ that is not the root, the {\it parent} of $t$ is the unique vertex $p$ such that $t$ is a child of $p$.
A tree-decomposition $(T,\X)$ of a graph is a {\it rooted tree-decomposition} if $T$ is a rooted tree. For a vertex set $S\subseteq V(G)$, an \emph{$S$-rooted tree-decomposition} is a rooted tree-decomposition $(T,\X)$ such that the bag $X_r$ at the root $r$ of $T$ is equal to $S$.

Let $(T,\X)=(T,(X_t:t\in V(T)))$ be a tree-decomposition of a graph $G$ and, for each $i$ in a finite set $I$, let $(T^i,\X^i)=(T^i,(X_t^i:t\in V(T^i))$ be a tree-decomposition of a graph $G_i$ such that $V(G)\cap V(G_i)$ is a bag of both $(T,\X)$ and $(T^i,\X^i)$, say $X_{z^i}$ and $X_{r^i}^i$ respectively. 
Suppose further that $V(G_i)-V(G)$ and $V(G_j)-V(G)$ are disjoint for all distinct $i,j\in I$.
Let $T^*$ be the tree obtained from the disjoint union $T \cup \bigcup_{i\in I}T_i$ by identifying $z^i$ with $r^i$ for each $i\in I$, and let $\X^*=(X_t^*:t\in V(T^*))$ where $X_t^*=X_t$ if $t\in V(T)$ and $X_t^*=X_t^i$ if $t\in V(T^i)$ for some $i\in I$.
It is easy to see that $(T^*,\X^*)$ is well-defined and it is a tree-decomposition of $G\cup \bigcup_{i\in I}G_i$. We say that $(T^*,\X^*)$ is obtained by \emph{attaching $(T^i,\X^i)$ to $(T,\X)$ along $X_{z^i}$ for each $i\in I$}.

For a graph $F$, we denote by $F^+$ the graph obtained from $F$ by adding a new vertex, called the \emph{apex}, adjacent to every vertex in $F$.

A \emph{separation} of a graph $G$ is an ordered pair $(A,B)$ with $A,B\subseteq V(G)$ such that $A\cup B=V(G)$ and there does not exist an edge of $G$ with one end in $A-B$ and the other in $B-A$.

We say that a graph $G$ contains a graph $H$ as a \emph{minor} if a graph isomorphic to $H$ can be obtained from $G$ by a sequence of vertex deletions, edge deletions, and edge contractions.

\section{Apex-forests}
To prove Theorem \ref{thm:apexforest} for apex-forests $H$, we may assume without loss of generality, by possibly adding edges to $H$, that $H=F^+$ for some tree $F$.
Given a graph $G$ that does not contain $F^+$ as a minor, we will show that $G$ has tree-width at most $|V(F)|-1$ by inductively building a rooted tree-decomposition $(T,\X)$ of $G$ such that the subgraph of $G$ induced by the root bag $S$ contains a spanning subgraph isomorphic to a subtree of $F$. This tree-decomposition will not necessarily have width at most $|V(F)|-1$, but we will require that every bag of size greater than $|V(F)|$ is a leaf bag and has a parent bag of size at most $|S|-1$. 

Let $G$ be a graph, let $S \subseteq V(G)$, and let $w$ be a nonnegative integer.
An {\it $(S,w)$-octopus} of $G$ is an $S$-rooted tree-decomposition $(T,(X_t: t \in V(T)))$ of $G$ such that
    \begin{itemize}
        \item for every $t \in V(T)$ with $|X_t| \geq w+1$, $t$ is a non-root leaf and $|X_p| \leq |S|$, where $p$ is the parent of $t$.
    \end{itemize}
An $(S,w)$-octopus of $G$ is {\it trivial} if every bag is either a subset of $S$ or equal to $V(G)$.
A {\it wrist} of an $(S,w)$-octopus $(T,(X_t: t \in V(T)))$ is a node $t$ of $T$ such that $|X_c| \geq w+1$ for some child $c$ of $t$; a wrist $t$ is {\it thick} if $|X_t|=|S|$.
An $(S,w)$-octopus is {\it thin} if it has no thick wrist.

\begin{lemma} \label{lem:thinoctopuseasy}
Let $G$ be a graph and let $S$ be a subset of $V(G)$. 
If $N_G(V(C)) \subsetneq S$ for every component $C$ of $G-S$, then there exists a thin $(S,|S|)$-octopus of $G$.
\end{lemma}

\begin{proof}
For each component $C$ of $G-S$, let $T_C$ be a rooted tree consisting of the root $r_C$ with one child $c_C$, and let $X_{r_C}=N_G(V(C))$ and $X_{c_C}=V(C)$.
Define $T$ to be the rooted tree obtained from the disjoint union of $T_C$ over all components $C$ of $G-S$ by adding the new root $r$ and edges such that each $r_C$ is a child of $r$.
Let $X_r=S$.
Then $(T,(X_t: t \in V(T)))$ is an $(S,|S|)$-octopus of $G$ such that every wrist is $r_C$ for some component $C$ of $G-S$.
Note that $|X_{r_C}|=|N_G(V(C))| \leq |S|-1$ for every component $C$ of $G-S$ by the assumption that $N_G(V(C))\subsetneq S$.
Hence, $(T,(X_t: t \in V(T)))$ is a thin $(S,|S|)$-octopus of $G$.
\end{proof}

\begin{lemma} \label{lem:thinoctopusbase}
Let $F$ be a tree.
Let $G$ be a graph that does not contain $F^+$ as a minor and let $S$ be a subset of $V(G)$ such that $G[S]$ contains a spanning subgraph isomorphic to $F$.
Then there exists a thin $(S,|V(F)|)$-octopus of $G$.
\end{lemma}

\begin{proof}
If there exists a component $C$ of $G-S$ such that $V(C)$ is adjacent to every vertex in $S$, then the graph obtained from $G[S \cup V(C)]$ by contracting $C$ into a single vertex contains a spanning subgraph isomorphic to $F^+$, contradicting the assumption that $G$ does not contain $F^+$ as a minor.
So for every component $C$ of $G-S$, we have $N_G(V(C)) \subsetneq S$.
Since $|S|=|V(F)|$, there exists a thin $(S,|V(F)|)$-octopus of $G$ by Lemma \ref{lem:thinoctopuseasy}.
\end{proof}
 
We will use Lemma \ref{lem:thinoctopusbase} as the base case in our induction to show that, whenever $G[S]$ contains a spanning subgraph isomorphic to a subtree of $F$, there exists a thin $(S,|V(F)|)$-octopus of $G$. Note that if $|S|=1$, then a thin $(S,|V(F)|)$-octopus of a connected graph $G$ is in fact a tree-decomposition of $G$ of width $|V(F)|-1$, as desired.

\begin{lemma} \label{lem:outopuslinked}
Let $G$ be a graph and let $S$ be a subset of $V(G)$.
Let $w$ be a nonnegative integer.
Let $(T,\X)$ be a non-trivial $(S,w)$-octopus of $G$ such that the number of thick wrists is as small as possible.
Denote $\X$ by $(X_t: t \in V(T))$.
Then for every thick wrist $t$ of $(T,\X)$, there exist $|S|$ disjoint paths in $G$ from $S$ to $X_t$.
\end{lemma}

\begin{proof}
Suppose to the contrary that there exists a thick wrist $z$ of $(T,\X)$ such that there does not exist $|S|$ disjoint paths in $G$ from $S$ to $X_z$.
Note that $|X_z|=|S|$ since $z$ is thick.
By Menger's theorem, there exists a separation $(A,B)$ of $G$ with $|A \cap B| \leq |S|-1$ such that $S \subseteq A$ and $X_z \subseteq B$, and such that there exist $|A \cap B|$ disjoint paths $P_1,P_2,...,P_{|A \cap B|}$ in $G$ from $S$ to $X_z$.
Note that for every $i \in [|A \cap B|]$, we have $|V(P_i) \cap A \cap B|=1$; we denote by $v_i$ the unique vertex in $V(P_i) \cap A \cap B$.

For every $t \in V(T)$, let $$X_t' = (X_t \cap A) \cup \{v_i: X_t \cap B \cap V(P_i) \neq \emptyset, i \in [|A \cap B|]\}.$$
Note that, since $X_z \subseteq B$ and $X_z \cap V(P_i) \neq \emptyset$ for every $i \in [|A \cap B|]$, we have 
\begin{align*}
    X'_z &= (X_z \cap A) \cup \{v_i: X_z \cap B \cap V(P_i) \neq \emptyset, i \in [|A \cap B|]\} \\
    &= (X_z \cap A \cap B) \cup (A \cap B) = A \cap B.
\end{align*}
Let $T'$ be the rooted tree obtained from $T$ by adding a child $c$ of $z$. Let $X'_c=B$ and $\X'=(X_t': t \in V(T'))$.

    \begin{claim} \label{lem:octopuslinkedtree:claim1}
        $(T',\X')$ is a tree-decomposition of $G$.  
    \end{claim}

    \begin{subproof}[Proof of Claim \ref{lem:octopuslinkedtree:claim1}]
Let $e\in E(G)$. If $e \in E(G[B])$, then $X'_c$ contains both ends of $e$, so we may assume that $e \not\in E(G[B])$. 
Since $(T,\X)$ is a tree-decomposition of $G$, there exists $t_e \in V(T)$ such that $X_{t_e}$ contains both ends of $e$. 
Since $e\not\in E(G[B])$, both ends of $e$ are in $A$, so $X'_{t_e} \supseteq X_{t_e} \cap A$ contains both ends of $e$.
Similarly, we have $\bigcup_{t \in V(T')}X'_t \supseteq V(G)$.

Now for each $v\in V(G)$, define $Y_v = \{t \in V(T'): v \in X'_t\}$.
To show that $(T',\X')$ is a tree-decomposition of $G$, it suffices to show that $T'[Y_v]$ is connected for all $v\in V(G)$.

If $v \in A-B$, then $Y_v = \{t \in V(T): v \in X_t\}$ induces a connected subgraph of $T'$ since $(T,\X)$ is a tree-decomposition of $G$.
If $v \in B-A$, then $Y_v=\{c\}$ induces a connected subgraph of $T'$.
So we may assume that $v=v_i$ for some $i \in [|A \cap B|]$.
This implies that
\begin{align*}
Y_v &= \{t \in V(T): v_i \in X_t\} \cup \{t \in V(T): X_t \cap B \cap V(P_i) \neq \emptyset\} \cup \{c\}\\
&= \{t \in V(T): X_t \cap B \cap V(P_i) \neq \emptyset\} \cup \{c\},   
\end{align*}
where the last inclusion holds since $v_i \in B \cap V(P_i)$.
Since $B \cap V(P_i)$ induces a connected subgraph of $G$ and $(T,\X)$ is a tree-decomposition of $G$, the set $\{t \in V(T): X_t \cap B \cap V(P_i) \neq \emptyset\}$ induces a connected subgraph of $T$.
Since $X_z \subseteq B$, $z \in \{t \in V(T): X_t \cap B \cap V(P_i) \neq \emptyset\}$.
Therefore $Y_v = \{t \in V(T): X_t \cap B \cap V(P_i) \neq \emptyset\} \cup \{c\}$ induces a connected subgraph of $T'$.
    \end{subproof}

In fact, $(T',\X')$ is an $S$-rooted tree-decomposition of $G$; indeed, since $X_r=S \subseteq A$, we have $X_r = X_r \cap A$, so 
\begin{align*}
X'_r &= (X_r \cap A) \cup \{v_i: X_r \cap B \cap V(P_i) \neq \emptyset, i \in [|A \cap B|]\} \\
&= X_r \cup \{v_i: X_r \cap A \cap B \cap V(P_i) \neq \emptyset, i \in [|A \cap B|]\} = X_r.    
\end{align*}
Since $(T,\X)$ is an $(S,w)$-octopus of $G$, we have $S=X_r=X_r'$.

To show that $(T',\X')$ is an $(S,w)$-octopus, it remains to show that for every $t\in V(T')$ with $|X_t'|\geq w+1$, we have that $t$ is a non-root leaf and $|X_p'|\leq |S|$, where $p$ is the parent of $t$.

    \begin{claim} \label{lem:octopuslinkedtree:claim2}
        For every $t \in V(T)$, we have $|X'_t| \leq |X_t|$. 
    \end{claim}

    \begin{subproof}[Proof of Claim \ref{lem:octopuslinkedtree:claim2}]
    Let $t \in V(T)$.
    For every $v \in X'_t-X_t$, we have $v=v_i$ for some $i \in [|A \cap B|]$ such that $v_i \not \in X_t \cap A$ and $X_t \cap B \cap V(P_i) \neq \emptyset$, so there exists $v_i' \in X_t \cap B \cap V(P_i)-A \subseteq X_t-X'_t$.
    Hence there exists an injection from $X'_t-X_t$ to $X_t-X_t'$.
    Therefore, $|X'_t| = |X'_t \cap X_t| + |X'_t-X_t| \leq |X'_t \cap X_t| + |X_t-X_t'| = |X_t|$.
    \end{subproof}

Let $t \in V(T')$ with $|X'_t| \geq w+1$.
By Claim \ref{lem:octopuslinkedtree:claim2}, we have either $t=c$ or $|X_t| \geq w+1$.
If $t=c$, then $t$ is a non-root leaf in $T'$, $z$ is the parent of $t$, and $|X'_z|=|A \cap B| < |S|$.
If $t \neq c$, then $t \in V(T)$ and $|X_t| \geq w+1$, so $t$ is a non-root leaf of $T$ and hence a non-root leaf of $T'$, and the parent $p$ of $t$ satisfies $|X'_p| \leq |X_p| \leq |S|$, since $(T,\X)$ is an $(S,w)$-octopus of $G$. 

Hence $(T',\X')$ is an $(S,w)$-octopus of $G$.
Note that $(T',\X')$ is non-trivial since $X_c'=B$ is a bag that is not a subset of $S$ nor equal to $V(G)$.
By Claim \ref{lem:octopuslinkedtree:claim2}, every thick wrist of $(T',\X')$ in $T$ is a thick wrist of $(T,\X)$.
Moreover, $z$ is a thick wrist of $(T,\X)$, but $z$ is not a thick wrist of $(T',\X')$ since $X'_z=A \cap B$ has size less than $|S|$.
Therefore, $(T',\X')$ is a non-trivial $(S,w)$-octopus of $G$ with fewer thick wrists than $(T,\X)$, a contradiction.
\end{proof}

\begin{lemma} \label{lem:thinoctopus}
Let $F$ be a tree and let $F'$ be a subtree of $F$.
Let $G$ be a graph that does not contain $F^+$ as a minor and let $S$ be a subset of $V(G)$ such that $G[S]$ contains a spanning subgraph isomorphic to $F'$.
Then there exists a thin $(S,|V(F)|)$-octopus of $G$.
\end{lemma}

\begin{proof}
We proceed by induction on the lexicographic order of $(|V(F)|-|V(F')|,|V(G)|)$.
The case $|V(F)|-|V(F')|=0$ follows from Lemma \ref{lem:thinoctopusbase}.
So we may assume that $F' \neq F$ and that the lemma holds when the lexicographic order of $(|V(F)|-|V(F')|,|V(G)|)$ is smaller.

If $|V(G)| \leq |V(F)|$, then the rooted tree-decomposition whose underlying tree has two nodes with root bag $S$ and the other bag $V(G)$ is a thin $(S,|V(F)|)$-octopus of $G$.
So we may assume that $|V(G)|>|V(F)|$.

Since $F' \neq F$, there exists an edge $uv \in E(F)$ such that $u \in V(F')$ and $v \in V(F)-V(F')$.
Let $F''=F'+uv$.
Since some spanning subgraph of $G[S]$ is isomorphic to $F'$, there exists an isomorphism $\phi$ from $F'$ to a spanning subgraph of $G[S]$.

By Lemma \ref{lem:thinoctopuseasy}, if $N_G(V(C))\subsetneq S$ for every component $C$ of $G-S$, then there exists a thin $(S,|S|)$-octopus of $G$, and since $|S|\leq |V(F)|$, there exists a thin $(S,|V(F)|)$-octopus of $G$.
So we may assume that $N_G(V(C))=S$ for some component $C$ of $G-S$.
Then there exists a vertex $v^* \in V(G)-S$ such that $v^*\phi(u) \in E(G)$.
Let $S' = S \cup \{v^*\}$.
Then $G[S']$ contains a spanning subgraph of $G[S']$ isomorphic to $F''$.
Since $|V(F)|-|V(F'')| < |V(F)|-|V(F')|$, by the inductive hypothesis, there exists a thin $(S',|V(F)|)$-octopus $(T^1,\X^1)$ of $G$.

Let $T^2$ be the rooted tree obtained from $T^1$ by adding a new root $r^2$ adjacent to the root of $T^1$.
Let $X^2_{r^2}=S$ and, for every $t \in V(T^1)$, let $X^2_t=X^1_t$.
Let $\X^2=(X^2_t: t \in V(T^2))$.
Since $|S|=|S'|-1$, $(T^2,\X^2)$ is a (not necessarily thin) $(S,|V(F)|)$-octopus of $G$.
Since $|V(G)|>|V(F)| \geq |S'|$, we have that $S'$ is a bag of $(T^2,\X^2)$ that is not a subset of $S$ nor equal to $V(G)$.
So $(T^2,\X^2)$ is a non-trivial $(S,|V(F)|)$-octopus of $G$.

Let $(T^3,\X^3)$ be a non-trivial $(S,|V(F)|)$-octopus of $G$ such that the number of thick wrists is as small as possible, and subject to this, $|V(T^3)|$ is as small as possible.
Denote $\X^3$ by $(X^3_t: t \in V(T^3))$.
Let $W$ be the set of thick wrists of $(T^3,\X^3)$.
For every $t \in W$, let $Q_t$ be the set of children $c$ of $t$ with $|X^3_c| \geq |V(F)|+1$.

By Lemma \ref{lem:outopuslinked}, for every $t \in W$, there exists a set $\PP_t$ of $|S|=|X^3_t|$ disjoint paths in $G$ from $S$ to $X^3_t$.
For every $t \in W$ and $c \in Q_t$, let $G_{t,c}=G[X^3_t \cup X^3_c] \cup \bigcup_{P \in \PP_t}P \cup G[S]$, and let $G_{t,c}'$ be the graph obtained from $G_{t,c}$ by contracting each path in $\PP_t$ into its unique vertex in $X_t^3$. 
Then $G_{t,c}'[X^3_t]$ contains a spanning subgraph isomorphic to $F'$ for every $t \in W$ and $c \in Q_t$; moreover, $G_{t,c}'$ is a minor of $G$ and hence does not contain $F^+$ as a minor.

\begin{claim} \label{lem:thinoctopus:claim1}
We may assume that for every $t \in W$ and $c \in Q_t$, $|V(G_{t,c}')| < |V(G)|$.
\end{claim}

\begin{subproof}[Proof of Claim \ref{lem:thinoctopus:claim1}]   
Since $G_{t,c}'$ is obtained from the subgraph $G_{t,c}$ of $G$ by contracting each path in $\mathcal{P}_t$, we have $|V(G_{t,c}')| \leq |V(G)|$. 
If equality holds, then $V(G_{t,c}')=V(G)$, so $X^3_t=S$ and $X^3_c \supseteq V(G)-S$, which implies $X^3_{t'} \subseteq S$ for every $t' \in V(T^3)-\{t,c\}$.
Since $(T^3,\X^3)$ is non-trivial, we have $X^3_c \neq V(G)$, so restricting $(T^3,\X^3)$ to the two nodes $t$ and $c$ yields a non-trivial $(S,|V(F)|)$-octopus of $G$.
Since $(T^3,\X^3)$ was chosen to minimize $|V(T^3)|$, we have $V(T^3)=\{t,c\}$.
Since $X^3_c \supseteq V(G)-S$ and $X^3_c \neq V(G)$, there is a vertex $u \in S$ not in $X^3_c$. 
So no vertex of $G-S$ is adjacent in $G$ to $u$.
Hence $N_G(V(C)) \subsetneq S$ for every component $C$ of $G-S$.
By Lemma \ref{lem:thinoctopuseasy}, there exists a thin $(S,|S|)$-octopus of $G$, and since $|S| \leq |V(F)|$, there exists a thin $(S,|V(F)|)$-octopus of $G$.
\end{subproof}

By the inductive hypothesis, for every $t\in W$ and $c\in Q_t$, there exists a thin $(X^3_t,|V(F)|)$-octopus $(T^{t,c},\X^{t,c})$ of $G_{t,c}'$. Denote $\X^{t,c}$ by $(X_z^{t,c}:z\in V(T^{t,c}))$ and let $r^{t,c}$ denote the root node of $T^{t,c}$.
Let $(T^*,\X^*)$ be the $S$-rooted tree-decomposition of $G$ obtained by attaching $(T^{t,c},\X^{t,c})$ to $(T^3-\bigcup_{t\in W}Q_t,\X^3-\{X_c^3:t\in W,c\in Q_t\})$ along $X_t^3$ for each $t\in W$ and $c\in Q_t$.
It is easy to see that $(T^*,\X^*)$ is in fact a thin $(S,|V(F)|)$-octopus of $G$ since every bag of $(T^*,\X^*)$ of size at least $|V(F)|+1$ is a bag of $(T^{t,c},\X^{t,c})$ for some $t\in W$ and $c\in Q_t$ and $(T^{t,c},\X^{t,c})$ is a thin $(X_t^3,|V(F)|)$-octopus.
This completes the proof of the lemma.
\end{proof}

The following theorem proves Theorem \ref{thm:apexforest} since $|V(F)|=|V(F^+)|-1$.

\begin{theorem}
Let $F$ be a tree.
If $G$ is a graph that does not contain $F^+$ as a minor, then the tree-width of $G$ is at most $|V(F)|-1$.
\end{theorem}

\begin{proof}
Since the tree-width of $G$ is equal to the maximum of the tree-widths of its components, we may assume without loss of generality that $G$ is connected.
Let $S$ be a set consisting of a single vertex of $G$.
Then $G[S]$ is isomorphic to a subtree of $F$.
By Lemma \ref{lem:thinoctopus}, there exists a thin $(S,|V(F)|)$-octopus $(T,\X)$ of $G$.

We claim that $(T,\X)$ does not have any wrists. Suppose to the contrary; let $t$ be a wrist of $(T,\X)$ and let $c$ be a child of $t$ such that the bag at $c$ has size at least $|V(F)|+1$. Since $(T,\X)$ is thin, the bag at $t$ has size at most $|S|-1=0$.
Since $S \neq \emptyset$ is the root bag and $c$ is a non-root leaf whose bag is non-empty, this contradicts the assumption that $G$ is connected.

Since $(T,\X)$ has no wrists, every bag of $(T,\X)$ has size at most $|V(F)|$.
Hence, $(T,\X)$ is a tree-decomposition of $G$ with width at most $|V(F)|-1$.
\end{proof}

\section{Wheels}
For an integer $k\geq 3$, recall that a \emph{$k$-wheel} is a graph $C^+$ where $C$ is a cycle of length $k$. 
The following theorem implies Theorem \ref{thm:wheel} since if $H$ is a wheel, then $H$ is a $(|V(H)|-1)$-wheel.

\begin{theorem}
Let $k \geq 3$ be an integer.
Let $G$ be a graph that does not contain a $k$-wheel as a minor.
Let $C$ be either an edge of $G$ or a cycle in $G$ with $|V(C)| \leq k-1$.
Then there exists a $V(C)$-rooted tree-decomposition of $G$ with maximum bag size at most $\max\{\tfrac{3}{2}k-3,k\}$ (i.e.~width at most $\max\{\tfrac{3}{2}k-4,k-1\}$).
\end{theorem}

\begin{proof}
We proceed by induction on the lexicographic order of $(|V(G)|, |V(G)|-|V(C)|)$. 
If $|V(G)|\leq \max\{\tfrac{3}{2}k-3,k\}$, then the rooted tree-decomposition whose underlying tree has two nodes with root bag $V(C)$ and the other bag $V(G)$ is a $V(C)$-rooted tree-decomposition of $G$ with maximum bag size at most $\max\{\tfrac{3}{2}k-3,k\}$. So we may assume that $|V(G)|>\max\{\tfrac{3}{2}k-3,k\}$ and that the theorem holds when the lexicographic order of $(|V(G)|,|V(G)|-|V(C)|)$ is smaller.
    \begin{claim}\label{claim:wheelconnectivity}
        We may assume that $G$ is 2-connected, $G-V(C)$ is connected, and every vertex in $C$ is adjacent in $G$ to some vertex in $G-V(C)$.
    \end{claim}
    \begin{subproof}[Proof of Claim \ref{claim:wheelconnectivity}]
        For every block $B$ of $G$ not containing $C$, let $C_B$ be an edge in $B$; for the block $B$ of $G$ containing $C$, let $C_B=C$.
        If $G$ is not 2-connected, then by the inductive hypothesis, each block $B$ of $G$ admits a $V(C_B)$-rooted tree-decomposition $(T^B,\X^B)$ with maximum bag size at most $\max\{\tfrac{3}{2}k-3,k\}$. By taking the disjoint union of these tree-decompositions and adding, for each cut-vertex $v$, a new node with bag $\{v\}$ and an edge joining this node to a node of $T^B$ whose bag contains $v$ for each block $B$ of $G$ containing $v$, we obtain a $V(C)$-rooted tree-decomposition of $G$ with maximum bag size at most $\max\{\tfrac{3}{2}k-3,k\}$.
        So we may assume that $G$ is 2-connected.

        Suppose that $G-V(C)$ is not connected.
        Let $M_1,M_2,...,M_t$ be the components of $G-V(C)$, where $t \geq 2$.
        For each $i \in [t]$, $G[V(C)\cup V(M_i)]$ is a non-spanning subgraph of $G$; hence, $G[V(C) \cup V(M_i)]$ does not contain a $k$-wheel as a minor and, by the inductive hypothesis, admits a $V(C)$-rooted tree-decomposition $(T^i,\mathcal{X}^i)$ with maximum bag size at most $\max\{\tfrac{3}{2}k-3,k\}$.
        Identifying the roots of $(T^1,\mathcal{X}^1),\dots,(T^t,\mathcal{X}^t)$ (whose corresponding bags are all equal to $V(C)$), we obtain the desired $V(C)$-rooted tree-decomposition of $G$.
        So we may assume that $G-V(C)$ is connected.
        
        Suppose there exists a vertex $v$ in $C$ that is not adjacent in $G$ to any vertex in $G-V(C)$.
        Since $G$ is 2-connected, we have $|V(C)| \geq 3$, so $C$ is a cycle.
        Let $G_v$ and $C_v$ be the graph and the cycle or edge obtained from $G$ and $C$ by contracting an edge of $C$ incident to $v$, respectively.
        Then $G_v$ is a minor of $G$ and $|V(G_v)|<|V(G)|$; hence, $G_v$ does not contain a $k$-wheel as a minor and, by the inductive hypothesis, $G_v$ admits a $V(C_v)$-rooted tree-decomposition $(T',\mathcal{X}')$ with maximum bag size at most $\max\{\tfrac{3}{2}k-3,k\}$.
        Adding a new node adjacent to the root of $(T',\mathcal{X}')$ as the new root with corresponding bag $V(C)$, we obtain the desired $V(C)$-rooted tree-decomposition of $G$.
        So we may assume that every vertex in $C$ is adjacent in $G$ to some vertex in $G-V(C)$.
    \end{subproof}

Let $c=|V(C)|$.

\begin{claim}
    \label{claim:choosePnotwholegraph}
    We may assume that there exists an edge $v_1v_c$ of $C$ and a path $P'$ in $G$ between $v_1$ and $v_{c}$ internally disjoint from $V(C)$ with $|V(P')|\geq 3$ such that $V(C)\cup V(P')\neq V(G)$.
\end{claim}
\begin{proof}[Proof of Claim \ref{claim:choosePnotwholegraph}]
    Let $v_1v_c$ be an edge of $C$.
    By Claim \ref{claim:wheelconnectivity}, $G-V(C)$ is connected and every vertex in $C$ is adjacent in $G$ to some vertex in $G-V(C)$, so there exists a path $P_{1c}'$ in $G$ between $v_1$ and $v_c$ internally disjoint from $V(C)$ with $|V(P_{1c}')|\geq 3$.
Choose $v_1v_c$ and $P_{1c}'$ so that $P_{1c}'$ is as short as possible. Then $v_1$ has a unique neighbor in $P_{1c}'-V(C)$ and $P_{1c}'-V(C)$ is an induced path.

We are done if $V(C)\cup V(P_{1c}')\neq V(G)$, so we may assume $V(C)\cup V(P_{1c}')= V(G)$.
Let $u_1$ be the unique neighbor of $v_1$ in $P_{1c}'-V(C)=G-V(C)$.
Let $u_1,u_2,\dots,u_\ell$ denote the vertices of $P_{1c}'$ in this order.

If $c=2$, then the choice of $P_{1c}'$ implies that $G$ is a cycle, so there exists a $V(C)$-rooted path-decomposition of $G$ with maximum bag size at most $3 \leq \max\{\tfrac{3}{2}k-3,k\}$.
So we may assume $c \geq 3$ and that $C$ is a cycle.
Let $v_1,v_2,\dots,v_c,v_1$ denote the vertices of $C$ in this order.

By Claim \ref{claim:wheelconnectivity}, every vertex $v_i$ in $C$ has at least one neighbor in $G-V(C)=P_{1c}'-V(C)$. By our choice of $v_1v_c$ and $P_{1c}'$, we have for every edge $v_iv_{i+1}$ of $C$ (where $v_{c+1}=v_1$) that each $v_i$ and $v_{i+1}$ has a unique neighbor $u$ and $u'$ respectively in $G-V(C)$, and moreover $\{u,u'\}=\{u_1,u_\ell\}$. In other words, since $u_1$ is the unique neighbor of $v_1$, we have that for all $i\in[c]$, if $i$ is odd, then $u_1$ is the unique neighbor of $v_i$ in $G-V(C)$, and if $i$ is even, then $u_\ell$ is the unique neighbor of $v_i$ in $G-V(C)$.

Now the bags $V(C),V(C) \cup \{u_1\}, (V(C)-\{v_i: i \in [c], i$ is odd$\}) \cup \{u_1,u_\ell\}, \allowbreak \{u_1,u_2,u_\ell\}, \{u_2,u_3,u_\ell\},\dots,\{u_{\ell-2},u_{\ell-1},u_\ell\}$ form a $V(C)$-rooted path-decomposition of $G$ with maximum bag size at most $\max\{|V(C)|+1,3\} \leq k$, as desired.
\end{proof}

We choose $v_1, v_{c}$, and $P'$ satisfying Claim \ref{claim:choosePnotwholegraph} so that $|V(M)|$ is maximized, where $M$ is a largest (in terms of the number of vertices) component of $G-(V(C) \cup V(P'))$.
Let $P$ denote the path $P'-\{v_1,v_{c}\}$.
Let $A = V(P) \cap N_G(V(M))$.
Note that $A \neq \emptyset$ because $G-V(C)$ is connected by Claim \ref{claim:wheelconnectivity}.

Moreover, $C\cup P'$ contains a cycle with vertex set $V(C)\cup V(P') \supseteq N_G(V(M))$. Since $G$ does not contain a $k$-wheel as a minor, we have $|N_G(V(M))| \leq k-1$.

    \begin{claim} \label{claim_non_crossing}
        There does not exist a path $Q$ in $G-E(C)$  between two distinct vertices $x,y$ of $P'$ such that $Q$ is internally disjoint from $V(C) \cup V(P') \cup V(M)$ and some vertex in $A$ is an internal vertex of the subpath $P_{xy}$ of $P'$ between $x$ and $y$.
    \end{claim}

    \begin{subproof}[Proof of Claim \ref{claim_non_crossing}]
        Suppose to the contrary that such a path $Q$ exists.
        Let $P''$ be the path obtained from $P \cup Q$ by deleting the internal vertices of $P_{xy}$.
        Then $v_1,v_c$, and $P''$ satisfy Claim \ref{claim:choosePnotwholegraph} and there is a component of $G-(V(C) \cup V(P''))$ containing $M$ and a vertex in $A$, contradicting the maximality of $M$.
    \end{subproof}

A {\it closed interval} is a subpath $Q$ of $P$ between two distinct vertices of $A$ with length at least two such that no internal vertex of $Q$ is in $A$.
An {\it open interval} is the subpath of a closed interval induced by its internal vertices.
For every open interval $I$, let $Y_I$ be the component of $G-(V(C) \cup A)$ containing $I$; by Claim \ref{claim_non_crossing}, each $Y_I$ is disjoint from $V(M)$ and from $V(P)-V(I)$, and we have $N_G(V(Y_I)) \subseteq V(C) \cup A_I$, where $A_I$ is the set of endpoints of the unique closed interval containing $I$.
In particular, if $I_1$ and $I_2$ are distinct open intervals, then $Y_{I_1}$ and $Y_{I_2}$ are disjoint.

    \begin{claim} \label{claim_root_cycle_interval}
        For every open interval $I$, we have $2 \leq |N_G(V(Y_I))| \leq k-1$ and there exists a cycle $C_I$ in $G[V(C) \cup (V(P)-V(I)) \cup V(M)]$ 
        such that $N_G(V(Y_I))\subseteq V(C_I)$.
    \end{claim}

    \begin{subproof}[Proof of Claim \ref{claim_root_cycle_interval}]
        Since $G$ is 2-connected, we have $2 \leq |N_G(V(Y_I))|$.
        Note that there exists a path $M_I$ in $G$ between the two vertices in $A_I$ such that all internal vertices of $M_I$ are in $M$.
        If $C$ is an edge, then let $C'=C$; otherwise, let $C'=C-v_1v_c$.
        Then $C_I=C'+(P'-V(I))+M_I$ is a cycle in $G[V(C) \cup (V(P)-V(I)) \cup V(M)]$, and 
        $$N_G(V(Y_I))\subseteq V(C)\cup A_I\subseteq V(C)\cup (V(P')-V(I))\subseteq V(C_I).$$
        The graph obtained from $G[V(Y_I)\cup N_G(V(Y_I))] \cup C_I$ by contracting $Y_I$ into a single vertex contains a $|N_G(V(Y_I))|$-wheel.
        Since $G$ does not contain a $k$-wheel as a minor, we have $|N_G(V(Y_I))| \leq k-1$.
    \end{subproof}

For every open interval $I$, let $G_I = G[V(Y_I) \cup N_G(V(Y_I))]$. Let $C_I$ be a cycle in $G[V(C)\cup(V(P)-V(I))\cup V(M)]$ such that $N_G(V(Y_I))\subseteq V(C_I)$ as in Claim \ref{claim_root_cycle_interval}.
Let $C_I'$ be the cycle or the edge obtained from $C_I$ by contracting a subset of its edges so that $V(C_I')=N_G(V(Y_I))$; let $G_I'$ be the graph obtained from $G_I\cup C_I$ by contracting the same set of edges, so that $V(G_I')=V(G_I)$. Note that $G_I'$ is a minor of $G$, so $G_I'$ does not contain a $k$-wheel as a minor.

    \begin{claim} \label{claim_interval_decomp}
        For every open interval $I$, there exists a $N_G(V(Y_I))$-rooted tree-decomposition of $G_I'$ with maximum bag size at most $\max\{\tfrac{3}{2}k-3,k\}$.
    \end{claim}

    \begin{subproof}[Proof of Claim \ref{claim_interval_decomp}]
        By Claim \ref{claim_root_cycle_interval}, we have $|V(C_I')| \leq k-1$, and $C_I'$ is either a cycle or an edge.
        Since $V(M) \cap V(G_I) = \emptyset$, we have $|V(G_I')|<|V(G)|$.
        Since $G_I'$ does not contain a $k$-wheel as a minor, by the inductive hypothesis, $G_I'$ admits the desired $N_G(V(Y_I))$-rooted tree-decomposition.
    \end{subproof}

Let $G^* = G[V(C) \cup A]$, that is, $G^*$ is obtained from $G[V(C) \cup V(P)]$ by deleting every open interval. Note that $V(G^*)=V(C)\cup A$.

    \begin{claim} \label{claim_reduce_to_central_2}
        It suffices to show that there exists a $V(C)$-rooted tree-decomposition of $G^*$ with maximum bag size at most $\max\{\tfrac{3}{2}k-3,k\}$ such that for every component $Q$ of $G-V(G^*)$, there is a bag containing $N_G(V(Q))$. 
    \end{claim}

    \begin{subproof}[Proof of Claim \ref{claim_reduce_to_central_2}]
        Suppose that $G^*$ admits a $V(C)$-rooted tree-decomposition $(T,\X)$ as in the claim.
        Then for every component $Q$ of $G-V(G^*)$, there is a bag of $(T,\X)$ containing $N_G(V(Q))$; by possibly adding leaf nodes, we may assume without loss of generality that for every component $Q$ of $G-V(G^*)$, there is a bag $X_Q$ of $(T,\X)$ equal to $N_G(V(Q))$. 
        Note that $|N_G(V(Q))|\geq 2$ since $G$ is 2-connected.

        For every open interval $I$, there exists a $N_G(V(Y_I))$-rooted tree-decomposition $(T^I,\X^I)$ of $G_I'$ with maximum bag size at most $\max\{\tfrac{3}{2}k-3,k\}$ by Claim \ref{claim_interval_decomp}; let $X_r^I$ denote its root bag, which is equal to $N_G(V(Y_I))$.
        
        If $C$ is an edge, then let $C'=C \cup P'$; if $C$ is a cycle, then let $C'=(C-v_1v_{c})\cup P'$.

        For every component $Q$ of $G-V(G^*)$ such that $Q\neq Y_I$ for every open interval $I$, we know that $N_{G}(V(Q)) \subseteq V(G^*)= V(C) \cup A$ and $C'$ is a cycle disjoint from $Q$ such that $V(C)\cup A\subseteq V(C')$. Since $G$ does not contain a $k$-wheel as a minor, we have $|N_{G}(V(Q))| \leq k-1$; let $C_Q'$ be the cycle or the edge obtained from $C'$ by contracting a subset of its edges so that $V(C_Q')=N_G(V(Q))$; let $G_Q$ be the graph obtained from $G[V(Q) \cup N_G(V(Q))]\cup C'$ by contracting the same set of edges. 

        Note that $M$ is a component of $G-V(G^*)$ such that $M\neq Y_I$ for every open interval $I$, so $C_M'$ is defined.
        If $M$ is not the unique component of $G-V(G^*)$ or if $|V(C_M')|<|V(C')|$, then for every component $Q$ of $G-V(G^*)$ such that $Q\neq Y_I$ for every open interval $I$, by the inductive hypothesis, $G_Q$ admits a $N_G(V(Q))$-rooted tree-decomposition $(T^Q,\X^Q)$ with maximum bag size at most $\max\{\tfrac{3}{2}k-3,k\}$, and we let $X_r^Q$ denote its root bag, which is equal to $N_G(V(Q))$. 
        Then by attaching $(T^I,\X^I)$ and $(T^Q,\X^Q)$ to $(T,\X)$ along $X_I=X_r^I$ and  $X_Q=X_r^Q$ respectively for each open interval $I$ and each component $Q$ of $G-V(G^*)$ such that $Q\neq Y_I$ for every open interval $I$,  we obtain the desired $V(C)$-rooted tree-decomposition of $G$ with maximum bag size at most $\max\{\tfrac{3}{2}k-3,k\}$.
        
        So we may assume that $M$ is the unique component of $G-V(G^*)$ and that $|V(C_M')|=|V(C')|$.
        The former implies $V(G)=V(C) \cup V(P) \cup V(M)$, and the latter implies $A=V(P)$ and $V(C)\subseteq N_G(V(M))$, hence $V(P) \cup V(C) = N_G(V(M))$.
        Thus, $C'$ is a cycle on at most $k-1$ vertices. Moreover, we have $N_G(V(M))\subseteq V(C')$ and $|V(C')| = |V(C)|+|A| > |V(C)|$, so $|V(G)|-|V(C')|<|V(G)|-|V(C)|$. 
        By the inductive hypothesis, $G$ admits a $V(C')$-rooted tree-decomposition with maximum bag size at most $\max\{\tfrac{3}{2}k-3,k\}$.
        Since $V(C) \subseteq V(C')$, $G$ admits a $V(C)$-rooted tree-decomposition with maximum bag size at most $\max\{\tfrac{3}{2}k-3,k\}$.
    \end{subproof}

By Claim \ref{claim_reduce_to_central_2}, it suffices to show that there exists a tree-decomposition of $G^*$ with maximum bag size at most $\max\{\tfrac{3}{2}k-3,k\}$ such that some bag contains $V(C)$ and for every component $Q$ of $G-V(G^*)$, there is a bag containing $N_G(V(Q))$. 
If $|A|=1$, then $|V(G^*)| = |V(C)|+1 \leq k$, and the single bag $V(G^*)$ forms such a tree-decomposition of $G^*$.

So we may assume $|A|\geq 2$. 
Let $a$ be the vertex in $A$ closest to $v_1$ on the path $P'$, and let $a'$ be the vertex in $A$ closest to $v_{c}$ on $P'$. Since $|A|\geq 2$, we have $a\neq a'$.

A {\it jump} is a path in $G$ from a vertex in $V(C)$ to a vertex in $A$ internally disjoint from $V(C) \cup A$.
Note that a jump may be a single edge.

For every $r \in V(C)$, let \[S_r = \{x \in A:\text{ there exists a jump $Q$ from $r$ to $x$ such that }Q \cap M=\emptyset\}.\]
Let $S = \{x \in V(C): S_x \neq \emptyset\}$.
Note that $\{v_1,v_{c}\}\subseteq S$ since the two edges of $P'$ incident with $\{v_1,v_c\}$ are both jumps.
We say that a vertex $r$ in $C$ is {\it bad} if $S_r \not \subseteq \{a,a'\}$.
\begin{claim}
    \label{claim:badneighbornotS}
    Let $r\in V(C)$ and let $r'$ be a neighbor of $r$ in $C$. If $r$ is bad, then $S_{r'}=\emptyset$.
\end{claim}
\begin{subproof}[Proof of Claim \ref{claim:badneighbornotS}]
    Suppose to the contrary that $r$ is bad and $S_{r'}\neq\emptyset$. 
    Then there exists $b\in A-\{a,a'\}$ and $b'\in A$ such that there exists a jump $Q$ from $r$ to $b$ and a jump $Q'$ from $r'$ to $b'$ such that $Q\cap M=Q'\cap M=\emptyset$. Then, in the union of $Q, Q'$, and the subpath of $P$ between $b$ and $b'$, there is a path $P^*$ between $r$ and $r'$ disjoint from $M$ such that there is a component of $G-V(P^*)$ containing $M$ and at least one vertex in $\{a,a'\}$, contradicting the maximality of $M$ in our choice of $v_1,v_c$, and $P'$.
\end{subproof}
In particular, no two bad vertices are adjacent in $C$.

    \begin{claim} \label{claim_bad_counts}
        There are at most $\max\{0,\lceil \frac{k-5}{2} \rceil\}$ bad vertices in $C-N_G(V(M))$.
    \end{claim}

    \begin{subproof}[Proof of Claim \ref{claim_bad_counts}]
        Since $\{v_1,v_c\} \subseteq S$, by Claim \ref{claim:badneighbornotS}, neither $v_1$ nor $v_c$ is bad and the neighbors of $v_1$ or $v_c$ not in $\{v_1,v_c\}$ are also not bad. 
        In particular, there are no bad vertices when $|V(C)|\leq 4$.
        So we may assume $|V(C)|\geq 5$ and that the bad vertices are contained in a subpath of $C$ on at most $|V(C)|-4$ vertices.
        Since no two bad vertices are adjacent in $C$ by Claim \ref{claim:badneighbornotS}, it follows that there are at most $\lceil \frac{|V(C)|-4}{2} \rceil \leq \lceil \frac{k-5}{2} \rceil$ bad vertices in $C-N_G(V(M))$.
    \end{subproof}

We now construct a tree-decomposition of $G^*$ satisfying the conditions of Claim \ref{claim_reduce_to_central_2}.

First suppose that either $k \geq 8$ or $|V(C)| \leq k-2$.
Let $T$ be the path on four nodes $t_1, t_2, t_3, t_4$ in this order. Let
\begin{align*}
    X_{t_1}&= V(C)\\
    X_{t_2}&= V(C)\cup \{a,a'\}\\
    X_{t_3}&= \{r\in V(C): r\text{ is bad or }r\in N_G(V(M))\} \cup \{a,a'\}\\
    X_{t_4}&= \{r\in V(C): r\text{ is bad or }r\in N_G(V(M))\} \cup A
\end{align*}
Let $\X=(X_{t_1},X_{t_2},X_{t_3},X_{t_4})$. Then $(T,\X)$ is a tree-decomposition of $G^*$. 
Indeed, since $V(G^*)=V(C)\cup A$, we have $\bigcup_{i\in[4]}X_t = V(G^*)$. 
Let $e$ be an edge of $G^*$ and let $r,x$ denote its ends. 
If $r,x\in V(C)$, then $X_{t_1}$ (and $X_{t_2})$ contains both ends of $e$. If $r,x\in A$, then $X_{t_4}$ contains both ends of $e$. So we may assume $r\in V(C)$ and $x\in A$. 
If $x\in \{a,a'\}$, then $X_{t_2}$ contains both ends of $e$; otherwise, we have $x \not\in \{a,a'\}$, hence $r$ is bad and $X_{t_4}$ contains both ends of $e$.
Lastly, it is easy to see that for every vertex $v\in V(G^*)$, the set $\{t\in V(T):v\in X_t\}$ induces a subpath of $T$.

We now show that $(T,\X)$ has maximum bag size at most $\max\{\tfrac{3}{2}k-3,k\}$.
We have $|X_{t_1}| =|V(C)|\leq k-1$.
If $k \geq 8$, then $|X_{t_2}| \leq k+1 \leq \tfrac{3}{2}k-3$; 
if $|V(C)| \leq k-2$, then $|X_{t_2}| \leq k$.
Note that $X_{t_3}\subseteq X_{t_4}$ since $\{a,a'\}\subseteq A$.
By Claim \ref{claim_bad_counts}, and since $A=V(P)\cap N_G(V(M))$ and $N_G(V(M))\subseteq V(C)\cup V(P)$,  we have 
\begin{align*}
    |X_{t_3}|\leq |X_{t_4}|&\leq \max\{0,\lceil \tfrac{k-5}{2} \rceil\} + |V(C)\cap N_G(V(M))| + |V(P)\cap N_G(V(M))|\\
    &= \max\{0,\lceil \tfrac{k-5}{2} \rceil\} + |N_G(V(M))|\\
    &\leq \max\{0,\lceil \tfrac{k-5}{2} \rceil\}+k-1\\
    &\leq \tfrac{3}{2}k-3.
\end{align*}
By the definition of jumps, for every component $Q$ of $G-V(G^*)$, the bag $X_{t_4}$ contains $N_G(V(Q))$. 
Therefore, we obtain a desired tree-decomposition of $G^*$.

Now suppose that $k \leq 7$ and $|V(C)|=k-1$.
So $|V(C)| = k-1 \leq 6$.

If there exists $a'' \in \{a,a'\}$ such that there exists $v \in V(C)-N_G(V(M))$ with $S_v= \{a''\}$, then we modify the above tree-decomposition by changing $X_{t_1}$ to be $V(C) \cup \{a''\}$ and changing $X_{t_2}$ to be $(V(C) \cup \{a,a'\})-\{v\}$, to obtain a tree-decomposition of $G^*$ with maximum bag size at most $\max\{k,\tfrac{3}{2}k-3\}$ such that some bag contains $V(C)$ and for every component $Q$ of $G-V(G^*)$, there is a bag containing $N_G(V(Q))$. 

So we may assume that for every $v \in V(C)-N_G(V(M))$, either $v$ is bad or $S_v = \{a,a'\}$.
As shown in the proof of Claim \ref{claim_bad_counts}, no vertex in $\{v_1,v_{c}\} \cup N_{C}(\{v_1,v_{c}\})$ is bad.
By our choice of $v_1,v_c$, and $P'$, we know that $S_{v} \neq \{a,a'\}$ for every $v \in \{v_1,v_{c}\} \cup N_{C}(\{v_1,v_{c}\})$.
Thus we have $\{v_1,v_{c}\} \cup N_{C}(\{v_1,v_{c}\}) \subseteq N_G(V(M))$, which implies $|N_G(V(M))\cap V(C)|\geq\min\{4,|V(C)|\}=\min\{4,k-1\}$. Since $a\neq a'$ and $|N_G(V(M))|\leq k-1\leq 6$, we have
\begin{align*}
    2\leq |A|
    = |N_G(V(M))|-|N_G(V(M))\cap V(C)|
    \leq k-1 - \min\{4,k-1\}
    \leq 2,
\end{align*}
hence equality holds which implies $k=7$ and $|N_G(V(M))\cap V(C)|=4$.

This implies $|V(C)|=k-1=6$, $|V(C)-N_G(V(M))|=2$, and that the two vertices in $V(C)-N_G(V(M))$ are adjacent.
Since $|A|=2$, there is no bad vertex, so the two vertices $x,x'$ in $V(C)-N_G(V(M))$, which are adjacent in $C$, satisfy $S_x=\{a,a'\}=S_{x'}$.
The union of two jumps from $x$ and $x'$ to $a$ contains a path $P^*$ between $x$ and $x'$ disjoint from $M$ such that there is a component in $G-V(P^*)$ containing $M$ and $a'$, contradicting the maximality of $M$ in our choice of $v_1, v_c$ and $P'$. 
This completes the proof of the theorem. 
\end{proof}

\bigskip

\bigskip

\noindent{\bf Acknowledgement:}
This paper was partially written when the first author visited Institute of Mathematics, Academia Sinica, Taiwan hosted by Bruce Reed. 
He is thankful for their hospitality.

\bibliography{ref}
\bibliographystyle{plain}

\end{document}